\documentclass[a4paper,12pt]{article}
\usepackage{hyperref}
\usepackage{amsmath}
\usepackage{amssymb}
\usepackage{amsthm}
\usepackage{mathrsfs}
\newtheorem{theorem}{Theorem}[section]

\newtheorem{proposition}[theorem]{Proposition}

\newtheorem{remark}[theorem]{Remark}

\newcommand{\EE}{\mathbb{E}}

\newcommand{\PP}{\mathbb{P}}

\newcommand{\tT}{\mathbb{T}}
\newcommand{\ZZ}{\mathbb{Z}}
\newcommand{\A}{{\cal A}}

\newcommand{\E}{{\cal E}}

\newcommand{\K}{{\cal K}}
\newcommand{\I}{{\cal I}}

\newcommand{\xP}{P}

\newcommand{\aB}{{\mathbb{B}}}
\newcommand{\aD}{{\mathbb{D}}}

\newcommand{\aY}{{\mathbb{Y}}}
\newcommand{\aG}{{\mathbb{G}}}

\newcommand{\aV}{{\mathbb{V}}}
\newcommand{\aP}{{\bf P}}

\newcommand{\hS}{{\mathbb{S}}}

\newcommand{\Ta}{{\cal T}}
\newcommand{\aTa}{{\bf{\Ta}}}

\newcommand{\yX}{{\cal X}}
\newcommand{\xX}{X}
\newcommand{\axX}{{\mathbb{X}}}

\newcommand{\oth}{{\overline{\theta}}}

\newcommand{\hP}{P}

\newcommand{\tpi}{{\widetilde{\pi}}}

\begin{document}

\title{\itshape Entropy of absorbed Markov chains}

\author{Servet Mart\'inez}

\maketitle

\begin{abstract}
We consider a strictly substochastic matrix or
an stochastic matrix with absorbing states.
By using quasi-stationary distributions 
one shows there is a canonical associated stationary 
Markov chain.
Based upon $2-$stringing representation of the 
resurrected chain we supply a stationary representation 
of the killed and the absorbed chains.
The entropies of these representations has a clear meaning 
when one identifies the probability measure of 
natural factors. 
The balance between the entropies of these 
representations and the canonical chain, serves to check
the suitability of the whole construction.
\end{abstract}

\bigskip


\bigskip


\section{Introduction}
Our starting point is an irreducible strictly substochastic matrix $P_I$ 
matrix on a countable set $I$. It defines a killed Markov chain when 
adding a cemetery that is an absorbing state. One of our 
purposes is to explore 
how one can study the entropy of this chain. 

\medskip

The problem can be set for a Markov chain that
is absorbed in a class of states, which is not necessarily a singleton.
It is in this enlarged setting that we study the entropy.
In this study we use the following concepts: 

\smallskip

\noindent - the quasi-stationary distribution
(q.s.d.) of the matrix $P_I$. 
It exists and it is unique when $I$ is finite  
and in the general case we assume there exists some q.s.d.;

\smallskip 

\noindent - the Markov chain defined by resurrecting the
absorbed chain with the q.s.d.;

\smallskip

\noindent- the $2-$stringing of the resurrected Markov chain.

\medskip

In Proposition \ref{prop1} we show that every q.s.d. defines
a canonical stationary distribution associated to an absorbed chain.
A construction of this associated stationary chain is given 
in Proposition \ref{prop2} showing that it 
can be retrieved from the resurrected chain 
with some additional random elements: the killing on the orbit, the 
hitting of the absorbing stated, and a walk on the absorbing states. 

\medskip

The $2-$stringing of the resurrected Markov chain is used  
to supply stationary Markov representations of the killed and the 
absorbed Markov chains in a proper way, to compute their entropies 
and give a clear meaning to them. 
This is done in Sections \ref{ssk} and \ref{ssa} and in Propositions   
\ref{prop3} and \ref{prop30}. The entropies are interpreted by identifying 
the probability measure on the fibers of some natural factors. 
The entropy of the killed chain is the entropy of the resurrected
chain plus the entropy of being alive or killed,
and in the absorbed case one must add the entropy of the states
where they are absorbed. These additional terms are
given by the Abramov-Rokhlin formula on some factors.

\medskip

Finally, in Proposition \ref{prop4}
the entropy of the associated stationary chain is decomposed into the 
entropies of the absorbed chain 
and of the walk on the absorbing states. This is an
element which serves to the understanding of 
the stationary representation of the absorbed chain.

\medskip

We note that since the killed and the absorbed trajectory are finites, 
then almost all the orbits of the stationary representations of the
killed and the absorbed Markov
chains contain all the killed or absorbed trajectories.

\medskip

As usual one uses the capital letter $H$ for the entropy of a finite random
variable and $h$ is the entropy of an stationary chain. 

\medskip

Even if it is not usual, we use trajectory to refer to a visit of a 
finite sequence of states; and orbit for 
a bilateral sequence of states, that is for a point in a bilateral 
product space.

\section{Killed and Absorbed chains}
\label{sec0}
Let $P_I=(P(i,j): i,j\in I)$ be an irreducible strictly substochastic 
matrix on a countable set $I$. As usual one adds
a state $\partial\not\in I$ called a cemetery and the extension  
of $P_I$ to $I\cup \{\partial\}$ is noted by $\hP$, it satisfies 
$\hP(i, \partial)=1-\sum_{j\in I} P(i,j)$ for $i\in I$,
and $\hP(\partial, \partial)=1$. Strictly substochasticity is equivalent 
to $\sum_{i\in I}\hP(i,\partial)>0$.
By irreducibility the states in $I$ are transient and $\partial$
is absorbing. The process defined by $P_I$ is identified with the 
chain absorbed at a unique cemetery $\partial$.
  
\medskip

The existence of a unique cemetery models
the killing when this is a phenomenon common to all states, 
for instance in extinction where 
a unique $\partial$ has a clear meaning. But there 
can be several ways of being killed or hitting a boundary, 
and this is expressed by the existence of a set of 
absorbing states which is not necessary a singleton.

\medskip

So, we consider a more general situation. Let
$\hP=(\hP(a,b): a,b\in I\cup \E)$ be a stochastic 
matrix on the countable set $I\cup \E$ such that its 
restriction to $I$ is $P_I=(\hP(i,j):i,j\in I)$, 
all the states in $\E$ are absorbing,  
$\hP(\epsilon,\epsilon)=1$ for $\epsilon\in \E$, 
and $\E$ is attained 
from $I$, that is $\sum_{i\in I} \hP(i,\E)>0$. 
We retrieve the one point absorption when $\E=\{\partial\}$.

\medskip

Let $\yX=(\yX_n)$ be a Markov chain with transition matrix $\hP$,
it will be called absorbed chain.
By $\PP_a$ we mean the law of this chain when starting from
$a\in I\cup \E$ and $\EE_a$ denotes the associated mean 
expected value. Let
$$
\tau_\E=\inf\{n\ge 1: \yX_n\in \E\}
$$
be the first return time to $\E$. If $I$ is finite the hypotheses 
made on the chain imply $\PP_i(\tau_\E<\infty)=1$ for all $i\in I$.
In the countable case one assumes 
$\PP_i(\tau_\E<\infty)=1$ for all $i\in I$.

\medskip

We call $\yX^{(k)}=(\yX_n: 0\le n<\tau_\E)$ the killed trajectory
and $\yX^{(a)}=(\yX_n: 0\le n\le \tau_\E)$ the absorbed trajectory,
both starting from $\yX_0$.
The first one finishes when it is killed and the second one in the
state where it is absorbed.

\subsection{Quasi-stationary distributions}

A quasi-stationary distribution (q.s.d.)
$\mu=(\mu(i): i\in I)$ associated to $P_I$ is 
a probability measure $\mu$ on $I$ such that 
$$
\forall i\in I: \quad \PP_\mu(\yX_n=i \, | \, \tau_\E>n)=\mu(i).
$$
By writing this equality for $n=1$, one checks that the row vector
$\mu^t$ is a strictly positive left eigenvector of $P_I$ 
properly normalized (summing up-to $1$), with
eigenvalue $\gamma=\PP_\mu(\tau_\E>1)\in (0,1)$, that is
\begin{equation}
\label{equ1}
\mu^t P_I=\gamma \, \mu^t \hbox{ with }
\gamma=\sum_{i,j\in I}\mu(i)P(i,j)=\PP_\mu(\tau_\E>1).
\end{equation}
It follows that $\PP_\mu(\tau_\E>k)=\gamma^k$ for all $k\ge 0$.
So, if $\mu$ is a q.s.d. then the survival time 
is Geometric($1-\gamma$), see Lemma 2.2 in \cite{fmp}.
In the finite case there is a unique q.s.d. (see \cite{das}) 
and it corresponds to the
normalized left Perron-Frobenius eigenvector and $\gamma$ is the 
associated eigenvalue. 
The properties of q.s.d. depend on the killed trajectory
$\yX^{(k)}=(\yX_n: 0\le n<\tau_\E)$.
In the infinite case q.s.d.'s 
can exist or not (because the positive left eigenvectors 
can be of infinite mass), and when they exist there 
can be several q.s.d.
In the sequel we fix some q.s.d. $\mu$ 
which, as just discussed, exists in the finite case and in the 
infinite case one assumes its existence.

\medskip

Let us give some independence properties between the time of
killing and the absorption state. In Theorem 2.6 in \cite{cmsm}
it was stated the independence relation 
$\PP_\mu(\yX_n=i, \tau_\E>n)=\mu(i) \gamma^n$ for all
$i\in I$ and $n\ge 0$.
Let us prove that when starting from $\mu$ then the pair
$(\yX_{\tau_\E-1},\yX_{\tau_\E})$ consisting in the last visited
state before absorption and the absorption state,  
is independent of the random time $\tau_\E$. 
For $n\ge 1$, $i\in I$, $\epsilon\in \E$, one has
\begin{eqnarray*}
&{}&\PP_\mu(\yX_{\tau_\E}-1=i, \yX_{\tau_\E}=\epsilon, \tau_\E=n)=
\PP_\mu(\yX_{n-1}=i, \yX_n=\epsilon, \tau_\E=n)\\
&{}&\; =\PP_\mu(\yX_n=\epsilon \, | \, \yX_{n-1}=i)\PP_\mu(\yX_{n-1}=i, 
\tau_\E>n-1)\\
&{}&\; =\hP(i,\epsilon)\, \mu(i) \, \PP_\mu(\tau_\E>n-1)
= \hP(i,\epsilon) \, \mu(i) \, \gamma^{n-1}.
\end{eqnarray*}
Then, the independence relation follows. We can be more precise,
we have
$$
\PP_\mu(\yX_{\tau_\E}-1=i, \yX_{\tau_\E}=\epsilon)
=\hP(i,\epsilon)\, \mu(i) (\sum_{l\ge 1}\gamma^{l-1})
=\hP(i,\epsilon) \, \mu(i) \, (1-\gamma)^{-1}.
$$
Since $\PP_\mu(\tau_\E=n)=(1-\gamma)\gamma^{n-1}$, then the 
desired relation holds
$$
\PP_\mu(\yX_{\tau_\E}-1=i, \yX_{\tau_\E}=\epsilon, \tau_\E=n)=
\PP_\mu(\yX_{\tau_\E}-1=i, \yX_{\tau_\E}=\epsilon)\PP_\mu(\tau_\E=n).
$$

The above computations also show that the exit law of $I$ when 
starting from $\mu$ satisfies
\begin{equation}
\label{equ2}
\PP_\mu(\yX_{\tau_\E}\!=\!\epsilon)\!=\!(1\!-\!\gamma)^{-1}
(\sum_{i\in I}\mu(i)\hP(i,\epsilon)), \;
\PP_\mu(\yX_{\tau_\E}\!=\!\epsilon,\tau_\E\!=\! n)\!=\!
(\sum_{i\in I}\mu(i)\hP(i,\epsilon)) \gamma^{n-1}.
\end{equation}
These properties depend on the 
absorbed trajectory $\yX^{(a)}=(\yX_n: 0\le n\le \tau_\E)$.

\section{Associated stationary chain}
We are considering the matrix $\hP(a,b)=(\hP(a,b): 
a,b\in I\cup \E)$.
Let $\rho=(\rho(a): a\in I\cup \E)$ be a probability vector.
One defines the matrix $\hP^\rho=(P^\rho(a,b): 
a,b\in I\cup \E)$ by
$$
\hP^\rho(i,b)=\hP(i,b) \hbox{ if } i\in I, b\in I\cup \E , \;\;
\hP^\rho(\epsilon,b)=\rho(b) \hbox{ if } \epsilon\in \E,
b\in I\cup \E.
$$
So, in $\hP^\rho$ the $\epsilon-$row 
is $\hP^\rho(\epsilon,\bullet)=\rho^t$ for all $\epsilon\in \E$.

\medskip

\begin{proposition}
\label{prop1}
Every q.s.d. $\mu$ of $P_I$ determines a 
probability distribution $\pi=(\pi(a): a\in I\cup \E)$ 
given by
\begin{equation}
\label{equ3}
\forall \epsilon\in \E, \; \pi(\epsilon)=\sum_{i\in I}
\mu(i)\hP(i,\epsilon), \hbox{ and }
\forall i\in I, \; \pi(i)=\gamma \mu(i).
\end{equation}
which is a stationary distribution of the matrix 
$\hP^\pi=(\hP^\pi(a,b): a,b\in I\cup \E)$. In a reciprocal way,
every distribution $\tpi$ that satisfies $\tpi^t=\tpi^t \hP^\tpi$
is defined by a q.s.d. $\mu$ as in (\ref{equ3}). So, if $P_I$
has a unique q.s.d. (as in the finite case) then there is a unique 
distribution $\pi$ that satisfies $\pi^t=\pi^t \hP^\pi$.  
\end{proposition}

\begin{proof}
The q.s.d. $\mu$ satisfies $\mu^t P_I=\gamma\mu^t$
with $\gamma\in (0,1)$ and $\sum_{i\in I} \mu(i)=1$. 
The vector $\pi$ is a probability 
distribution because from (\ref{equ3}) and (\ref{equ1}) 
one gets that $\pi(I)=\sum_{i\in I} \pi(i)$ and 
$\pi(\E)=\sum_{\epsilon\in \E}\pi(\epsilon)$ satisfy
\begin{equation}
\label{equ4}
\pi(I)=\sum_{i\in I}
\mu(i)\hP(i,\epsilon)=\gamma \hbox{ and }
\pi(\E)=\sum_{i\in I} \mu(i)\hP(i,\E)=1-\gamma.
\end{equation}
Let us check that $\pi$ is stationary for $P^\pi$.
For $\epsilon\in \E$ and $j\in I$ we have
\begin{eqnarray*}
&{}& (\pi^t \hP^\pi)(\epsilon)=\pi(\epsilon)\sum_{\delta\in \E}
\pi(\delta)+\sum_{i\in I} \pi(i) \hP(i,\epsilon)=  
\pi(\epsilon)(1-\gamma)+\gamma \pi(\epsilon)=\pi(\epsilon)\, \\
&{}& (\pi^t \hP^\pi)(j)=\pi(j)\sum_{\epsilon\in \E}
\pi(\epsilon)+\sum_{i\in I} \pi(i) P(i,j)=
\pi(j)(1-\gamma)+\gamma \pi(j)=\pi(j).
\end{eqnarray*}
Then $\pi^t=\pi^t \hP^\pi$ holds.

\medskip

Now, we check that $\tpi$ is a
probability distribution that satisfies $\tpi^t=\tpi^t \hP^\tpi$,
then it is defined by a q.s.d. $\mu$ as in (\ref{equ3}).
For $j\in I$ one has,
$$
\tpi(j)=(\tpi^t \hP^\tpi)(j)=\tpi(j) \sum_{\delta\in \E}
\tpi(\delta)+ \sum_{i\in I} \tpi(i) P(i,j).
$$
and so $\tpi(j)(1-\sum_{\delta\in \E}
\tpi(\delta))=\sum_{i\in I} \tpi(i) P(i,j)$.
Then, the restriction $\tpi_I=(\tpi(i): i\in I)$ satisfies
$\gamma \tpi_I^t= \tpi_I^t P_I$
for some $\gamma$. So, $\tpi_I$ is an strictly positive 
left eigenvector with finite mass, then $\mu=\alpha^{-1} \tpi_I$ is a
q.s.d. when one takes $\alpha=\sum_{i\in I} \tpi(i)$ and its eigenvalue 
is $\gamma$. So $\tpi_I=\gamma \mu=(\pi(i): i\in I)$ given by 
the second term in (\ref{equ3}). On the other hand one gets
$$
\tpi(\epsilon)=(\tpi^t \hP^\tpi)(\epsilon)=\tpi(\epsilon)\sum_{\delta\in 
\E}
\tpi(\delta)+\sum_{i\in I} \tpi(i) \hP(i,\epsilon).
$$
Then,
$\tpi(\epsilon)(1-\sum_{\delta\in \E}
\tpi(\delta))=\sum_{i\in I} \tpi(i) \hP(i,\epsilon)$, so
$\tpi(\epsilon)\alpha=\alpha \sum_{i\in I} \mu(i) \hP(i,\epsilon)$
giving the equality $\tpi(\epsilon)=\sum_{i\in I} \mu(i)
\hP(i,\epsilon)$, so $\tpi=\pi$ which finishes the proof.
\end{proof}

Observe that (\ref{equ2}) is written
$$
\PP_\mu(\yX_{\tau_\E}=\epsilon)=\frac{\pi(\epsilon)}{\pi(\E)}=
\pi(\epsilon \, | \, \E) \hbox{ and }
\PP_\mu(\yX_{\tau_\E}=\epsilon,\tau_\E=n)=\pi(\epsilon \, | \, \E)
(1-\gamma) \gamma^{n-1}.
$$

We note by $\axX=(\xX_n)$ the Markov chain evolving with 
the transition kernel $\hP^\pi$ and call it the associated 
stationary chain. Since $\hP^\pi$ extends $\hP$, 
by abuse of notation we also note it by $P$. All the concepts 
developed in the absorbed case depended only on the trajectory 
$\yX^{(a)}=(\yX_n: n\le \tau_\E)$ which is equally distributed 
as $(\xX_n: n\le \tau_\E)$ when starting 
from $X_0=\yX_0\in I$. Hence, there is no confusion if 
one continues noting by $\PP_a$ the 
law of the chain $\axX$ starting from $a\in I\cup \E$ and by 
$\EE_a$ its associated mean expected value. 

\medskip

Since $\axX=(\xX_n)$ has transition probability kernel $\xP$
and stationary distribution $\pi$, its entropy is
\begin{eqnarray}
\nonumber
h(\axX)&=&-\sum_{\delta\in \E} \pi(\delta)
\sum_{a\in I\cup \E}\pi(a) \log  \pi(a)
-\sum_{i\in I} \pi(i) \sum_{b\in I\cup \E} P(i,b) \log  P(i,b)\\
\label{equ21}
&=& -\pi(\E) \sum_{i\in I} \pi(i) \log  \pi(i)-
\pi(\E) \sum_{\delta\in \E} \pi(\delta) \log \pi(\delta)\\
\nonumber
&{}&\; -\sum_{i,j\in I} \pi(i) P(i,j) \log  P(i,j)
-\sum_{i\in I,\delta\in \E} \pi(i) P(i,\delta) \log P(i,\delta) .
\end{eqnarray}

Further, we will compare this entropy to the entropies of 
some random sequences appearing in the chain. 

\section{Elements of the associated stationary chain}

The object of this section is to show how one can retrieve 
the chain $\axX$ from the 
absorbed trajectories and some walks on the absorbing states.
In this purpose, the behavior of chain $\axX$ is firstly decomposed   
along its visits to $I$ and to $\E$ in a separated way.

\subsection{Decoupling the stationary chain}

Let
$$
\tau_I=\inf\{n\ge 1: \xX_n\in I\}
$$
be the first return time of $\axX$ to $I$. Now,
consider the stochastic matrix $Q=(Q(i,j): i,j\in I)$ given by
$$
Q(i,j)=\xP_i(X_{\tau_I}=j).
$$
By using that 
$\PP_\epsilon(\xX_{\tau_I}=j)=\pi(j \, | \, I)=\mu(i)$ for 
all $\epsilon\in \E, j\in I$, one gets
\begin{eqnarray}
\nonumber
Q(i,j)&=& \EE_i \left({\bf 1}_{\{\xX_{\tau_I}=j, {\tau_I}=1\}}\right)
+\EE_i\left({\bf 1}_{\{\xX_{\tau_I}=j,{\tau_I}>1\}}\right)\\
\nonumber
&=& P(i,j)+ \PP_i({\tau_I}>1) 
\PP_i\left(\xX_{\tau_I}=j \, | \, \xX_{\tau_I-1}, {\tau_I}>1 \right)\\
\label{equ30}
&=&P(i,j)+P(i,\E) \mu(j).
\end{eqnarray}
Let $\aY=(Y_n: n\in \ZZ)$ be a Markov chain with transition matrix $Q$.
It is straightforwardly checked that $\mu$ is a stationary measure 
for $\aY$.

\begin{remark}
\label{rmk1}
For a substochastic matrix $P_I$ the matrix 
$Q=(Q(i,j)=P(i,j)+P(i,\E) \mu(j): i,j\in I)$ was defined 
in \cite{fkmp} and called the resurrected 
matrix from $P_I$ with distribution $\mu$. It was
a key concept used in \cite{fkmp} to prove
the existence of q.s.d. for geometrically absorbed Markov chains 
taking values in an infinite countable set. $\Box$
\end{remark}

The chain $\aY$ can be constructed as follows. Let
$\Xi=\{\xi_l: l\in \ZZ\}$ be the ordered sequence given by
$$
\{\xi_l: l\in \ZZ\}=\{n\in \ZZ: \xX_n\in I\} \hbox{ with } 
\xi_{l-1}<\xi_{l}\, , \; \xi_{-1}<0\le \xi_0.
$$
Then, $(\xi_{l}-\xi_{l-1}: l\in \ZZ)$ is a 
renewal stationary sequence with interarrival times distributed as
$\PP(\xi_{l}-\xi_{l-1}=\bullet)=\PP_\mu(\tau_I=\bullet)$, $l\neq 0$.
By definition $(\xX_{\xi_l}: l\in \ZZ\}$ is a stationary sequence
distributed as $\aY=(Y_n: n\in \ZZ)$,
so $(\xX_{\xi_l}: l\in \ZZ)$ is a copy of $\aY$.

\medskip

The random sequence 
${\bf{b}}=(b_l: l\in \ZZ, l\neq 0)$ defined by
$b_l=1$ if $\xi_{l}-\xi_{l-1}=1$ and $b_l=0$ if 
$\xi_{l}-\xi_{l-1}>1$, 
is a collection of i.i.d. random variables, with
$$
\PP(b_l=1)=\pi(I) \hbox{ and } 
\PP(b_l=0)=\sum_{i\in I} \mu(i) P(i,\E)=\pi(\E).
$$
When $\xX_0\in I$, one finds $\tau_\E =\inf\{l\ge 1: b_l=0\}$.

\begin{remark}
\label{rmk9}
Every irreducible matrix stochastic matrix $Q$ with stationary distribution 
$\mu$ can be written as in (\ref{equ30}). 
In fact, let $\chi=(\chi(i): i\in I)$ be a non-null vector,  
$\chi\neq {\vec 0}$, that satisfies
$$
\forall i\in I: \quad 0\le \chi(i)<1 \hbox{ and } \chi(i)\le \min\{Q(i,j)
\mu(j)^{-1}: j\in I\}.
$$
This can be achieved because $\mu$ is strictly positive.
Define $P_I=Q-\chi \mu^t$, so
\begin{equation}
\label{equ31}
\forall i,j\in I: \quad P(i,j)=Q(i,j)-\chi(i) \mu(j).
\end{equation}
To avoid trivial situation one can take $\chi$ also satisfying that
for every $i\in I$ and for some (or for all) $j\in I$ for which $Q(i,j)>0$ one
has $P(i,j)>0$. This allows to take $\chi$ ensuring $P_I$ is irreducible. 
From the construction $P(i,j)\in [0,1)$ and since $\chi\neq 0$ we get
$$
\sum_{j\in I} P(i,j)=1-\chi(i)\in (0,1] \hbox{ and } \sum_{i, j\in I} (1-P(i,j))>0.
$$
Hence $P$ is strictly substochastic, it is not trivial, 
and when adding the cemetery $\partial$ one has 
$\chi(i)=P(i,\partial)$. So, 
$$
\mu^t P=\mu^t(Q-\chi\mu^t)=(1-\sum_{i\in I} \mu(i)P(i,\partial))\mu^t,
$$
that is $\mu^t$ is the Perron-Frobenius left eigenvector of $P$
with eigenvalue $\gamma=\sum_{i, j\in I} \mu(i) P(i,j)$, see (\ref{equ1}). 
From (\ref{equ31}) it follows $Q(i,j)=P(i,j)+P(i,\partial)\mu(j)$, so
(\ref{equ30}) is satisfied. $\Box$
\end{remark}

The restriction of $P=P^\pi$ to the absorption states $\E$ 
satisfies,
$$
\forall \epsilon,\delta\in \E: \;\; \xP(\epsilon,\delta)=\pi(\delta),
\hbox{ so }
\PP(X_1= \epsilon \, | \, X_0=\delta, X_1\in E)=\pi(\delta \, | \, \E).
$$
The transition law to an absorbing point after being in 
$\xX_{t-1}=i\in I$, is given by
$$
\forall \delta\in \E: \quad
\PP(\xX_{t}=\delta \, | \, \xX_{t}\in \E, \xX_{t-1}=i)
=P(i,\delta)/P(i,\E).
$$
So, if $X_{-1}\in I$ and $X_0\in \E$, the total sojourn time at $\E$ 
is $\tau_I$ satisfies
$\tau_I\sim \hbox{Geometric}(\gamma)$.
Then, immediately after the entrance to $\E$ the chain $\xX$ makes a 
walk on $\E$ of length $\tau_I-1$ (quantity that could vanish). To describe it
take $\aG=(G_n: n\in \ZZ)$ be a Bernoulli chain with
probability vector $\pi(\bullet \, | \, \E)$.
Let us consider an independent time 
$V\sim\, \hbox{Geometric}(\gamma)-1$ (that is 
$V+1 \sim\, \hbox{Geometric}(\gamma)$) and define a finite 
sequence $V=(G_l: 1\le l<\tau_I)$ which is distributed as follows,
\begin{eqnarray}
\label{equ5}  
\PP(V=\emptyset)&=&\PP(\tau_I=1)=\pi(I);\\
\nonumber
\PP(V=(\delta_1,..,\delta_{k-1}))&=&
\PP(G_1\!=\!\delta_1,..,G_{k-1}\!=\!\delta_{k-1}, 
\tau_I\!=\!k)\\
\nonumber
&=& (\prod_{l=1}^{k-1}\pi(\delta_l)/\pi(\E))\pi(\E)^{k-1}\pi(I) 
=(\prod_{l=1}^{k-1}\pi(\delta_l))\pi(I)\\
\nonumber
&{}&
\hbox{ for } k\ge 2, (\delta_1,..,\delta_{k-1})\in \E^{k-1}.
\end{eqnarray}
Notice that last equality also holds when $\tau_I=k=1$
because an empty product satisfies $\prod_{l=1}^{k-1}=1$. One has,
$$
(X_t, 1\le t < \tau_I \, | \, X_{-1}\in I, X_0\in \E)\sim V,
$$
and $V$ is called a walk on $\E$. Note that 
$\tau_I-1 \, | \, \tau_I>1$ is equally distributed as $\tau_I$.
The exit law from $\E$ is $\PP(\xX_{\tau_I}\in \bullet)\sim \mu$.
In fact, for all $\delta\in E$ it holds
\begin{eqnarray*}
\PP_\delta(\xX_{\tau_I}=i)&=&\sum_{\epsilon\in
\E}\PP_\delta(\xX_{\tau_I}=i,
(\xX_{\tau_I-1}=\epsilon)=\sum_{\epsilon\in \E}
\PP_\delta(\xX_{\tau_I-1}=\epsilon)
\frac{\xP_\epsilon(\xX_1=i)}{\xP_\epsilon(\xX_1\in I)}\\
&=& \pi(i \, | \, I)=\mu(i).
\end{eqnarray*}
Notice that
\begin{equation}
\label{equ6}
\forall \delta\in \E: \quad \EE_\delta(\tau_I)=\pi(I)^{-1}\,.
\end{equation}

We consider a sequence of i.i.d. random variables 
$\aTa=(\Ta_n: n\in \ZZ)$ which are $\hbox{Geometric}(\gamma)-1$ 
distributed, that is 
$\PP(\Ta_n=l)=\gamma (1-\gamma)^{l}$ for $l\ge 0$.
The construction of i.i.d. walks on $\E$ is made as follows. One 
takes an increasing sequence of
times $(t_n: n\in \ZZ)$ with $t_{n+1}-t_n=\Ta_n$ and such that
$t_n\to \infty$ if $n\to \infty$ and $t_n\to -\infty$ if 
$n\to -\infty$. One defines 
$$
V^n=(G_{t_n},..,G_{t_{n+1}-1})=(V^n_1,..,V^n_{\Ta_n}).
$$
So, $\aV=\left(V^n: n\ge \ZZ\right)$ is an i.i.d. sequence
of walks on $\E$. The walk $V^n$ is empty when $\Ta_n=0$.

\subsection{Retrieving the stationary chain} 
Let $\aY=(Y_n: n\in \ZZ)$ be a stationary Markov chain 
with transition matrix $Q$.
Our purpose is to construct a copy of $\axX$ from $\aY$ 
by adding a series of random operations.

\medskip

Let $\PP$ be a probability measure governing 
the law of $\aY$ when it starts from the stationary distribution 
$\mu$, the sequences $\aG$, $\Ta$ and so $\aV$, 
and also the random element $\aB^{I,I}$
and $\aD^I$ defined below.

\medskip

Let $\aB^{I,I}=\left((B^{i,j}_l: l\in \ZZ); i,j\in I \right)$ 
be an independent array of Bernoulli random variables such that
$B^{i,j}_l\sim B^{i,j}$ for $l\in \ZZ$, where 
\begin{eqnarray}
\label{equ7}
&{}& \PP(B^{i,j}=1)=\theta_{i,j}, \, \PP(B^{i,j}=0)
=\oth_{i,j}=1-\theta_{i,j} \hbox{ with }\\
\nonumber 
&{}& \theta_{i,j}=\frac{P(i,j)}{P(i,j)+P(i,\E)\mu(j)}
=\frac{P(i,j)}{Q(i,j)}.
\end{eqnarray}
Let 
\begin{equation}
\label{equ8}
\tau_\partial=\inf\{l\ge 1: B^{Y_{l-1},Y_l}_l=0\}.
\end{equation}
For $k\ge 1$, $i_0,..,i_{k-1}\in I$ one has,
\begin{eqnarray}
\label{equ9}
&{}&\PP(Y_0=i_0, Y_1=i_1,...,Y_{k-1}=i_{k-1}, \tau_\partial=k)\\
\nonumber
&{}&\; =
\mu(i_0)(\prod_{l=1}^{k-1} P(i_{l-1},i_l))
(\sum_{j\in I}P(i_{k-1},\E)\mu_j)
=\mu(i_0)(\prod_{l=1}^{k-1} P(i_{l-1},i_l))P(i_{k-1},\E). 
\end{eqnarray}
Hence, the distribution of the sequence $Y^{(k)}=(Y_l: 0\le l<\partial)$ 
is the one of the killed chain $\yX^{(k)}$ starting from $\mu$.

\medskip

Now take an independent array
$\aD^I=\left((D^{i}_l: l\in \ZZ); i\in I \right)$ 
of random variables with values in $\E$ and law
\begin{equation}
\label{equ10}
\forall \delta \in \E:\quad \PP(D^{i}_l=\delta)=P(i,\delta)/P(i,\E).
\end{equation}

For $k\ge 1$, $i_0,..,i_{k-1}\in I$, $\delta\in \E$ we set,
\begin{eqnarray*}
&{}&\PP(Y_0=i_0, Y_1=i_1,...,Y_{k-1}=i_{k-1}, D^{i_{k-1}}_k=\delta, 
\tau_\partial=k)\\
&{}&\; =
\mu(i_0)(\prod_{l=1}^{k-1} P(i_{l-1},i_l))P(i_{k-1},\E)
(P(i_{k-1},\delta)/P(i_{k-1},\E))\\
&{}&\; =\mu(i_0)(\prod_{l=1}^{k-1} P(i_{l-1},i_l))P(i_{k-1},\delta).
\end{eqnarray*}
Then, the distribution of the sequence 
$Y^{(a)}=(Y_0,..,Y_{\tau_\partial-1}, 
D^{Y_{\tau_{\partial}-1}}_{\tau_\partial})$ 
is the one of a the absorbed chain $\yX^{(a)}$ starting from $\mu$.
 
\medskip

Let us construct a chain $\hS^s=(\hS^s_t: t\in \ZZ)$ from $\aY$, 
$\aB^{I,I}$, $\aD^Y$, $\aG$ and $\Ta$ (and so also $\aV$),
having the same distribution as $\axX$. 
Firstly, define a random sequence $\hS=(\hS_t: t\in \ZZ)$ as follows.
We set $T_0=0$, $\hS_0=Y_0$ and,

\smallskip 

\noindent I. In a sequential way on $n\ge 0$, one makes the 
following construction. Assume
at step $n$, $T_n$ has been defined then one puts $\hS_{T_n}=Y_n$
and goes to step $n+1$;

\smallskip

\noindent I.a. If $B^{Y_n,Y_{n+1}}_{n+1}=1$ put $T_{n+1}=T_n+1$,  
$\hS_{T_{n+1}}=Y_{n+1}$ and go to step $n+2$;

\smallskip

\noindent I.b. If $B^{Y_n,Y_{n+1}}_{n+1}=0$ put 
$T_{n+1}=T_n+\Ta_n+2$, define $\hS_{T_{n}+1}=D^{Y_n}_{n+1}$, 
$\hS_{T_{n}+1+l}=V^n_l$ for $1\le l<\Ta_n$ (it is empty 
when $\Ta_n=0$) and $\hS_{T_{n+1}}=Y_{n+1}$. 
After one continues with step $n+2$.

\medskip

\noindent II. Similarly, in a sequential way on $n<0$ one makes the
following construction for step $n$;

\smallskip

\noindent IIa. If $B^{Y_{n},Y_{n+1}}_{n+1}=1$ put $T_{n}=T_{n+1}-1$,
$\hS_{T_{n}}=Y_n$ and continue with step $n-1$;

\smallskip

\noindent IIb. If $B^{Y_n,Y_{n+1}}_{n+1}=0$ put $T_{n}=T_{n+1}-(\Ta_n+2)$, 
$\hS_{T_n+1}=D^{Y_n}_{n+1}$, $\hS_{T_{n}+1+l}=V^n_l$ for
$1\le l<{\Ta_n}$ and
$\hS_{T_n}=Y_{n}$. After one continues with step $n-1$.

\medskip

Let $\hS=(\hS_t: t\in \ZZ)$ be the random sequence resulting from this
construction. 

\medskip

Let $\tT=(T_n: n\in \ZZ)$, recall that $T_0=0$. By abuse of notation
we also note by $\tT=\{T_n: n\in \ZZ\}$ the set of these values.
By definition $\tT=\{t\in \ZZ: \hS_t\in I\}$
is the set of random points where $\hS$ is in $I$. 
We have that $(\hS, \tT)$ is a regenerative process, see \cite{as} 
p. 169-170, that is for all $T_n$ one has that 
$(\hS_{\bullet+ T_l}: \bullet \ge 0; T_n-T_l, n\ge l)$ 
has the same distribution as $(\hS_{\bullet}: \bullet \ge 0; 
T_n, n\ge 0)$ and it is independent of $(T_n: n\le l)$. 

\medskip

We have that a cycle of this regenerative process is the sequence 
of states $(\hS_0,...,\hS_{\tau_I-1})$, in fact this is the
common distribution for $(\hS_{T_n},...,\hS_{T_{n+1}-1})$. 
By shifting the process $\hS^s=(\hS^s_t: t\in \ZZ)$ in a random 
time chosen uniformly in $\{T_0,..,T_1-1\}$, one 
gets a stationary process $\hS^s=(\hS^s_t: t\in \ZZ)$ (see Theorem 6.4 
in \cite{as}).

\begin{proposition}
\label{prop2}
The processes $\hS^s$ and $\axX$ are equally distributed.
\end{proposition}

The proof is done in the Appendix.

\medskip

So, if $\hS^s_0\in I$, for $\tau^{{\bf{\hS}^s}}_\E=\inf\{l\ge 1: 
B^{\hS^s_{l-1}, \hS^s_{l}}_l=0\}$ one has that $({\bf{\hS}^s})^{(a)}
=(\hS^s_n: n\le \tau^{{\bf{\hS}^s}}_\E)$ is a copy of the absorbed chain,
see (\ref{equ8}) and (\ref{equ9}).

\section{Stationary representation of killed and absorbed chains}

The stationary Markov chain $\aY=(Y_n)$ with transition matrix 
$Q$ and stationary distribution $\mu$ has entropy 
\begin{equation}
\label{equx1}
h(\aY)=-\sum_{i\in I} \mu(i) \sum_{j\in I} Q(i,j)\log Q(i,j).
\end{equation}

For the stationary representation of the killed and absorbed chains 
we use the $2-$stringing form
of $\aY$. Let us recall this notion. Consider the 
stochastic matrix $Q^{[2]}$ with set of indexes $I^2$ given by,
$$
Q^{[2]}((i,j),(l,k))=Q(j,l) {\bf 1}(j=l).
$$
Its stationary distribution satisfies $\nu((i,j))=\mu(i)Q(i,j)$  
for $(i,j)\in I^2$.
In fact, by using $\sum_{i\in I}\mu(i)Q(i,j)=\mu(j)$ one gets
$$
\sum_{(i,j)\in I^2}\nu((i,j)) Q^{[2]}((i,j),(l,k))
=\sum_{i\in I}\mu(i)Q(i,l)Q(l,k)=\mu(l)Q(l,k)=\nu((l,k)).
$$
The stationary chain $\aY^{[2]}=((Y^1_n,Y^2_n):
Y^2_{n-1}=Y^1_{n}, n\in \ZZ)$ evolving with $Q^{[2]}$
is the $2-$stringing of $\aY$. We write it by 
$\aY^{[2]}=((Y_{n-1},Y_n): n\in \ZZ)$. It is well-known that
it is conjugated to $\aY$ by the ($1-$coordinate) mapping
$$
\Upsilon(((Y_{n-1},Y_n): n\in \ZZ))=(Y_n: n\in \ZZ).
$$
(This is stated in a general form in Lemma $1$ in \cite{ks}).
Let us check it is measure preserving. Take
$(i_l: l=0,..,k)\in I^{k+1}$. We have
\begin{eqnarray*}
&{}&\PP((Y_{n-1},Y_n) \in \aY^{[2]}:
(Y_{l-1},Y_l)=(i_{l-1},i_l): l=1,..,k)\\
&=&\mu(i_0)Q(i_0,i_1)\prod_{l=1}^{k-1} Q(i_l,i_{l+1})
=\PP(Y_l\!=\!i_l : l\!=\!0,..,k).
\end{eqnarray*}
So, the orbits $((Y_{n-1},Y_n): n\in \ZZ)\in \aY^{[2]}$ by 
can be identify with the orbits $(Y_n: n\in \ZZ)\in \aY$. 

\subsection{The killed chain}
\label{ssk}

The stationary representation of the killed chain is a stationary 
Markov chain with set of states 
$I^2\times \{0,1\}=\{(i,j,a): (i,j)\in I^2, a\in \{0,1\}\}$.
To introduce its transition matrix $\K$ define, 
$$
\varphi(i,j,a)=(\theta_{i,j}{\bf 1}(a=1)
+\oth_{i,j}{\bf 1}(a=0)),\; (i,j,a)\in I^2\times \{0,1\}.  
$$
The matrix
$\K=\left(\K((i,j,a),(l,k,b)): (i,j,a), (l,k,b))\in I^2\times \{0,1\}\right)$
is given by
$$
\K((i,j,a), (l,k,b))\!=\!
\begin{cases}
\!\!\!\!\! &0 \! \hbox{ if } l\! \neq \! j;\\
\!\!\!\!\! &\! Q(j,k)\varphi(j,k,b)\!=\!P(j,k){\bf 1}(b\!=\!1)
\!+\! P(j,\E)\mu(k){\bf 1}(b\!=\!0)
\hbox{ if } l\!=\!j.
\end{cases}
$$
It is a stochastic matrix: We claim its stationary distribution
$\zeta=(\zeta(i,j,a): (i,j,a)\in I^2\times \{0,1\})$ is given by
\begin{equation}
\label{equ11}
\zeta(i,j,a)\!=\! \mu(i)Q(i,j)\varphi(i,j,a)
\!=\!\mu(i)P(i,j){\bf 1}(a=1)\!+\!\mu(i)P(i,\E)\mu(j){\bf 1}(a=0).
\end{equation}
By using $\sum_{i\in I} \mu(i) Q(i,l)=\mu(l)$ 
one gets the desired property, 
\begin{eqnarray*}
&{}&\sum_{(i,j,a)\in I^2\times \{0,1\}}\!\!\!\!
\zeta(i,j,a) \K((i,j,a), (l,k,b))= 
(\sum_{i\in I} \mu(i) Q(i,l)) Q(l,k)\varphi(l,k,b)\\   
&{}&\; =
\mu(l)Q(l,k)\varphi(l,k,b)=\zeta(l,k,b),  
\end{eqnarray*}
so the claim follows. 

\medskip

The killed Markov chain presented in its stationary form is noted  
$\aY^{(\K)}=((Y_{n-1},Y_n,B_n): n\in \ZZ)$, it 
takes values in $I^2\times \{0,1\}$ and has
transition matrix $\K$. Below the component $B_n$ is called 
the label at $n$. By hypothesis $P_I$ is irreducible 
so also $\K$ is an irreducible matrix. Then the Markov shift 
$\aY^{(\K)}$ is ergodic (see Proposition 8.12 in \cite{dgs}).

\medskip

It is straightforward to check that the mapping,
\begin{equation}
\label{equ12}
\Upsilon^{(\K)}: \aY^{(\K)}\to \aY, (((Y_{n-1},Y_n,B_n): n\in \ZZ))\to (Y_n: n\in \ZZ),
\end{equation}
is a measure preserving factor.

\medskip

The orbit $((Y_{n-1},Y_n,B_n): n\in \ZZ)$ 
in $\aY^{(\K)}$ is noted in the simpler form $((Y_n,B_n): n\in \ZZ)$.

\begin{remark} 
\label{rmk2}
Let us see $\aY^{(\K)}$ models the killed Markov chain.  
Let ${\cal N}=\{n\in \ZZ: B_n=0\}$ and write it as
${\cal N}=\{n_l: l\in \ZZ\}$ with 
$n_l$ increasing in $l$ and $n_{-1}<0\le n_0$. Note that
$\PP(0\in {\cal N})=\pi(\E)$.

\medskip

We divide an orbit in $\aY^{(\K)}$ into the disjoint connected  
pieces
$$
(Y,B)^{(k)}_{l}=((Y_{n_l},1),..,(Y_{n_{l+1}-2},1), (Y_{n_{l+1}-1}, 0)),
\; l\in \ZZ\,.
$$
The component $Y_{n_l}$ is distributed with law $\mu$ for all $l$,
and one can identify $(Y,B)^{(k)}_{l}$ with
$Y^{(k)}_{l}=(Y_{n_l},..,Y_{n_{l+1}-1})$ a piece of the orbit
$Y=(Y_n: n\in \ZZ)$ starting from $\mu$ at $n_l$ and killed at $n_{l+1}-1$. 
One gets,
$$
\forall l\neq 0: \quad Y^{(k)}_{l}\sim \yX^{(k)}.
$$
when $\yX^{(k)}$ starts from $\mu$.
In fact, for $s\ge 0$, $i_0,..,i_s\in I$ we have,
\begin{eqnarray*}
&{}&\PP(\yX^{(k)}=(i_0,..,i_s))=\mu(i_0)\prod_{r=0}^{s-1} P(i_r,i_{r+1}) 
P(i_l,\partial)\\
&{}&=\mu(i_0)\left(\prod_{r=0}^{s-1} Q(i_r,i_{r+1})\theta_{i_r,i_{r+1}}\right)
(\sum_{m\in I}Q(i_s,m)\oth_{i_s,m})=\PP(Y_{l}^{(k)}=(i_0,..,i_s)). 
\end{eqnarray*}
Let $s\ge 0$, $(i_0,..,i_s)\in I^{s+1}$.
For almost all the orbits $Y\in \aY^{(\K)}$ and for all $l\in \ZZ$,
$l\neq 0$, one has 
$$
\PP(n_0=0, Y_0^{(k)}=(i_0,..,i_s))=\pi(\E)
\mu(i_0)\prod_{r=0}^{s-1} P(i_r,i_{r+1}) P(i_s,\partial)>0.
$$
Since each killed trajectory is finite, the class of killed trajectories 
is countable. Therefore, from the Ergodic Theorem and since $\aY^{(\K)}$ is ergodic, 
we get that $\PP-$a.e. the orbits of $\aY^{(\K)}$
contain all the killed trajectories of the chain. $\Box$
\end{remark}

The entropy of the killed chain is
\begin{eqnarray*}
&{}& \!\!\!\!\! h(\aY^{(\K)})\\
&=& \!-\!\sum_{(i,j)\in I^2} \!\!\! \mu(i) Q(i,j) \sum_{k\in I}
Q(j,k)\left(\theta_{j,k} \log (Q(j,k)\theta_{j,k})
+ \oth_{j,k}\log (Q(j,k)\oth_{j,k})\right)\\
&=& \!- \!\sum_{(j,k)\in I^2} \!\!\! \mu(j) Q(j,k) \log Q(j,k)
-\sum_{(j,k)\in I^2}\mu(j)Q(j,k) H(B^{j,k})
\end{eqnarray*}
where
$$
H(B^{j,k})=\theta_{j,k} \log \theta_{j,k}+\oth_{j,k} \log \oth_{j,k}
$$
is the entropy of the Bernoulli random variable $B^{j,k}$. Hence
\begin{equation}
\label{equ13}
h(\aY^{(\K)})=h(\aY)+\Delta(B) \hbox{ with }
\Delta(B)=-\sum_{(j,k)\in I^2}\mu(j)Q(j,k) H(B^{j,k}).
\end{equation}

The quantity $\Delta(B)=h(\aY^{(\K)})-h(\aY)$ is the conditional entropy of 
$\aY^{(\K)}$ given the factor $\aY$, see Lemma $2$ and 
Definition $3$ in \cite{dos}. To be more precise, given an orbit 
$Y=(Y_n: n\in \ZZ)$ of $\aY$, the fiber given by (\ref{equ12}) 
satisfies 
$(\Upsilon^{(\K)})^{-1}\{Y\}=\{(B^{Y_{n-1},Y_{n}}_n: n\in \ZZ)\in 
\{0,1\}^\ZZ\}$,
and it is distributed as a sequence of independent Bernoulli variables 
given by (\ref{equ7}), we note it 
by $\aP_{Y}$. We have
\begin{equation}
\label{equ14}
H_{\aP_{Y}}(B_1^{Y_0,Y_{1}})
=-(\theta_{Y_0,Y_1} \log \theta_{Y_0,Y_1}+ \oth_{Y_0,Y_1}\log 
\oth_{Y_0,Y_1}).  
\end{equation}

Let us summarize the results on the entropy of $\aY^{(\K)}$.

\begin{proposition}
\label{prop3}
The entropy of the stationary representation $\aY^{(\K)}$ of the killed 
chain satisfies
\begin{eqnarray}
\nonumber 
&{}&h(\aY^{(\K)})=h(\aY)+\Delta(B) \hbox{ with }\\ 
\label{equ15}
&{}&\Delta(B)=\int H_{\aP_{Y}}(B_1^{Y_0,Y_{1}}) d\PP(Y)=-
\sum_{(i,j)\in I^2}\mu(i)Q(i,j) H(B^{j,k});
\end{eqnarray}
and
\begin{eqnarray} 
\nonumber
 h(\aY^{(\K)}) &=&-\sum_{i\in I} \mu(i) P(i,\E)\log P(i,\E)-(1-\gamma)
\sum_{j\in I} \mu(j)\log \mu(j)\\
\label{equ16}
&{}&\; -\sum_{i,j\in I} \mu(i) P(i,j)\log P(i,j).
\end{eqnarray}
\end{proposition}

\begin{proof}
From (\ref{equ13}) and (\ref{equ14}), and by using the Markov property, 
ones retrieves
the Abramov-Rokhlin formula, see \cite{ar} and \cite{dos},
$$
\Delta(B)=h(\aY^{(\K)})-h(\aY)=\int H_{\aP_{Y}}(B_1^{Y_0,Y_1}) d\PP(Y)
=\sum_{i,j\in I} \mu(i) Q(i,j)H(B^{i,j}).
$$
This gives (\ref{equ15}).
The only thing left to prove is (\ref{equ16}). By using
$$
\sum_{i\in \E}\mu(i)P(i,\E)\!=\!1\!-\!\gamma,\,
\sum_{j\in I}P(i,j)\!=\!1\!-\!P(i,\E),\,
\sum_{i\in I}\mu(i) P(i,j)\!=\!\gamma \mu(j),
$$
and (\ref{equx1}) one gets,
\begin{eqnarray*}
\Delta(B)&=&
-\sum_{i,j\in I} \mu(i) \left(P(i,\E)\mu(j)\log(P(i,\E)\mu(j))+
P(i,j)\log P(i,j)\right)\\
&{}& \; +\sum_{i,j\in I} \mu(i) Q(i,j)\log Q(i,j)\\
&=&-\sum_{i\in I} \mu(i) P(i,\E) \log P(i,\E)-
(1-\gamma) \sum_{j\in I} \mu(j)\log \mu(j)\\
&{}& \; +\sum_{i,j\in I} \mu(i) P(i,j)\log P(i,j)-h(\aY).
\end{eqnarray*}
This shows (\ref{equ16}).
\end{proof}

\begin{remark}
\label{rmk4}
From (\ref{equ11}) we get that there are in mean
$\sum_{i\in I}\mu(i)P(i,j)=\gamma$ sites in $\ZZ$
where $\aY^{(\K)}$ makes a transition 
with label $0$, and a mean $\sum_{i\in I}\mu(i)P(i,\E)=
1-\gamma$ of sites in $\ZZ$ where $\aY^{(\K)}$ makes a transition 
with label $1$. $\Box$
\end{remark}

\subsection{The absorbed chain}
\label{ssa}

Let us construct a stationary representation of the absorbed 
chain in a similar way as  
we did for the killed chain. Define $\E^*=\E\cup \{o\}$ with
$o\not\in \E\cup I$. The absorbed chain will take values on the
set of states
$I^2\times \E^*=\{(i,j,\delta): i\in I, j\in I, \delta\in \E^*\}$.
The matrix $\A=(\A((i,j,\delta), (l,k,\epsilon)):
(i,j,\delta), (l,k,\epsilon))\in I^2\times \E$ defined by
$$
\A((i,j,\delta), (l,k,\epsilon)=
\begin{cases}
&0 \hbox{ if } l\neq j;\\
&Q(j,k)\theta_{j,k}=P(j,k) \hbox{ if } l=j, \epsilon=o;\\
&Q(j,k) \oth_{j,k} P(j,\epsilon)/P(j,\E)\!=\!P(j,\epsilon) \mu(k)
\hbox{ if } l\!=\!j, \epsilon\in \E;
\end{cases}
$$
is an stochastic matrix whose stationary distribution
$\eta=(\eta(i,j,\delta): (i,j,r)\in I^2\times \E^*)$ is given by
\begin{eqnarray*}
\eta(i,j,\delta)&=&
\mu(i)Q(i,j)\left(\theta_{i,j} {\bf 1}(\delta=o)+
\oth_{i,j} P(i,\delta)/P(i,\E) {\bf 1}(\delta\in \E)\right)\\
&=& \mu(i) P(i,j){\bf 1}(\delta=o)+\mu(i)P(i,\delta)\mu(j)
{\bf 1}(\delta\in \E).
\end{eqnarray*}
In fact, since $\sum_{\delta\in \E^*}(\theta_{i,l} {\bf 1}(\delta=o)+
\oth_{i,l} P(i,\delta)/P(i,\E) {\bf 1}(\delta\in \E))=1$ one gets the 
stationarity property,
\begin{eqnarray*}
&{}&
\sum_{(i,j,\delta)\in I^2\times \E^*} \eta(i,j,\delta)
\A((i,j,\delta), (l,k,\epsilon))\\
&{}& \, =
(\sum_{i\in I} \mu(i) Q(i,l))
Q(l,k)(\theta_{l,k}{\bf 1}(\epsilon=o)+\oth_{l,k} P(l,\epsilon)/P(l,\E)
{\bf 1}(\epsilon\in \E))\\ 
&{}& \, =
\mu(l) Q(l,k) (\theta_{l,k}{\bf 1}(\epsilon=o)+\oth_{l,k}
P(l,\epsilon)/P(l,\E) {\bf 1}(\epsilon\in \E))
=\eta(l,k,\epsilon).
\end{eqnarray*}   

We note by $\aY^{(\A)}=((Y_{n-1},Y_n, D^*_n): n\in \ZZ)$
the absorbed Markov chain presented in its stationary form,
and taking values in $I^2\times \E^*$ with transition matrix $\A$.
Since $\A$ is irreducible the Markov shift $\aY^{(\A)}$ is ergodic.

\medskip

It is straightforward to check that the mapping
\begin{eqnarray}
\nonumber
&{}& 
\Upsilon^{(\A)}: \aY^{(\A)}\to \aY^{(\K)},(((Y_{n-1},Y_n,D^*_n): n\in \ZZ))
\to (Y_{n-1},Y_n,B_n): n\in \ZZ)\\
\label{equ17}
&{}& \hbox{ with } B_n= {\bf 1}(D^*_n=o)\, ,
\end{eqnarray}
is a measure preserving factor between
$\aY^{(\A)}$ and $\aY^{(\K)}$.

\medskip

The orbit $((Y_{n-1},Y_{n},D^*_n): n\in \ZZ)$
in $\aY^{(\A)}$ can be written by $((Y_n,D^*_n): n\in \ZZ)$. 

\begin{remark}
\label{rmk5}
Let us see the stationary chain $\aY^{(\A)}$ models the absorbed Markov 
chain.
First note ${\cal N}^*=\{n\in \ZZ: D^*_n\in \E\}$ and write it by
${\cal N}^*=\{n_l: l\in \ZZ\}$ with $n_l$ is increasing in $l$ 
and $n_{-1}<0\le n_0$. We have $\PP(0\in {\cal N}^*)=\pi(\E)$.

\medskip

Similarly as we proceed in Remark \ref{rmk2}, an 
orbit $((Y_n,D^*_n): n\in \ZZ)\in \aY^{\A}$ is partitioned 
into the disjoint connected pieces 
$$
(Y,D^*)^{(a)}_{l}=((Y_{n_l},o),..,(Y_{n_{l+1}-2},o), 
(Y_{n_{l+1}-1}, D^*_{n_{l+1}-1})) \hbox{ with } l\in \ZZ\,.
$$ 
The component $Y_{n_{l}}$ is distributed with law $\mu$
for all $l$,
and one can identify $(Y,D^*)^{(a)}_l$ with 
$Y^{(a)}_l=(Y_{n_l},..,Y_{n_{l+1}-2}, D^*_{n_{l+1}-1})$
staring from $\mu$.
Since the events $\{n\in \ZZ: D^*_n\in \E\}$ has the same 
distribution as $\{n\in \ZZ: B_n=0\}$ in $\aY^{(\K)}$, 
one checks that
$$
\forall l\neq 0: \quad Y^{(a)}_{l}\sim \yX^{(a)},
$$
where $\yX^{(a)}$ starts form $\mu$.
In fact, for $s\ge 0$, $i_0,..,i_s\in I$, $\epsilon\in \E$, one has,
$$
\PP(\yX^{(a)}=(i_0,..,i_s, \epsilon))\!=\!
\mu(i_0)\left(\prod_{r=0}^{s-1}
P(i_r,i_{r+1})\right) 
P(i_s,\epsilon)\!=\!\PP(Y_{l}^{(a)}=(i_0,..,i_s,\epsilon)). 
$$
Let  $s\ge 0$, $(i_0,..,i_s)\in I^{s+1}$, $\epsilon\in \E$.
For almost all the orbits  $Y\in \aY^{(\A)}$ and for all $l\in \ZZ$, 
$l\neq 0$, one has
$$
\PP(0\in {\cal N}^*, Y_{l}^{(a)}=(i_0,..,i_l,\epsilon))=
\mu(i_0)\prod_{r=0}^l P(i_r,i_{r+1}) P(i_l,\epsilon)>0.  
$$
Since each absorbed trajectory is finite, the class of absorbed trajectories    
is countable. Then, form the Ergodic Theorem and since 
$\aY^{(\A)}$ is ergodic we get that $\PP-$a.e. the orbits of $\aY^{(\A)}$
contain all the absorbed trajectories of the chain. $\Box$
\end{remark}

The entropy of the absorbed chain is
\begin{eqnarray*}
&{}& \!\!\! h(\aY^{(\A)})\\
&=&-\sum_{(i,j)\in I^2}\!\!\!\! \mu(i) Q(i,j) \sum_{k\in I} \!
Q(j,k) \theta_{j,k} \log (Q(j,k)\theta_{j,k})\\
&{}& -\! \sum_{(i,j)\in I^2}\!\!\!\! \mu(i) Q(i,j) \!\!\!
\sum_{k\in I,\epsilon\in \E}\!\!\!\!
Q(j,k) \oth_{j,k} P(j,\epsilon)/P(j,\E) \log
(Q(j,k)\oth_{j,k}P(j,\epsilon)/P(j,\E))\\
&=& -\sum_{(j,k)\in I^2}\!\! \mu(j) Q(j,k) \log Q(j,k)
-\sum_{j\in I}\mu(j)Q(j,k)
(\theta_{j,k} \log \theta_{j,k}+\oth_{j,k} \log \oth_{j,k})\\
&{}& -\sum_{(j,k)\in I^2}\!\! \mu(j) Q(j,k)\oth_{j,k}
\sum_{\epsilon\in \E} P(j,\epsilon)/P(j,\E) \log (P(j,\epsilon)/P(j,\E)).
\end{eqnarray*}
Then,
\begin{eqnarray}
\nonumber
&{}&h(\aY^{(\A)})=h(\aY^{(\K)})+\sum_{i\in I}\mu(i) P(i,\E) H(D^i)\,, 
\hbox{ where }\\
\label{equ18}
&{}& H(D^i)=-\sum_{\delta\in \E}P(i,\delta)/P(i,\E) \log (P(i,\delta)/P(i,\E))
\end{eqnarray}
is the entropy of a random variable in $\E$ distributed as the
transition probability from $i\in I$ to an state conditioned to be in $\E$.
Note that the above expression can be also written
$$
h(\aY^{(\A)})=h(\aY^{(\K)})+\sum_{i\in I}\mu(i) (P(i,\E) H(D^i)+
P(i,\I))H(o)),
$$
being $H(o)=0$ the entropy of a constant.

\medskip

Define
$$
\Delta(D)=h(\aY^{(\A)})-h(\aY^{(\K)})=\sum_{i\in I}\mu(i) P(i,\E) H(D^i).
$$
This is the conditional entropy of
$\aY^{(\A)}$ given the factor $\aY^{(\K)}$. 
Let $Y^{(\K)}=((Y_{n-1},Y_{n},B_n): n\in \ZZ)$ be an orbit 
of $\aY^{(\K)}$. The fiber given by (\ref{equ17}) satisfies 
$(\Upsilon^{(\A)})^{-1}\{Y^{(\K)}\}=
\{(D^{Y_n, B_n}_n: n\in \ZZ)\in (\E^*)^\ZZ\}$ 
with $D^{Y_n, B_n}_n\in \E$ when $B_n=0$ 
and $D^{Y_n, B_n}_n=o$ when $B_n=1$. These variables
are independent distributed as a Bernoulli $D^{Y_n}_n$ given in 
(\ref{equ10}) if $B_n=0$ and the constant variable $o$ if 
$B_n=1$. This probability 
measure is noted by $\aP_{Y^{(\K)}}$. Thus, 
\begin{equation}
\label{equ19}
H_{\aP_{Y^{(\K)}}}(D^{Y_0,B_0}_0)
=
\begin{cases}
& \!\! -\sum_{\delta\in \E} P(Y_0,\delta)/P(Y_0,\E) \log 
(P(Y_0,\delta)/P(Y_0,\E)) \hbox{ if } B_0=0,\\
& \!\! 0 \hbox{ if } B_0=1.
\end{cases}
\end{equation}

\begin{proposition}
\label{prop30}
The entropy of the stationary representation $\aY^{\A}$ of the absorbed 
chain satisfies
\begin{eqnarray*}
&{}& h(\aY^{(\A)})=h(\aY^{(\K)})+\Delta(D) \hbox{ with }\\
&{}&\Delta(D)=\int H_{\aP_{Y}}(D^{Y_0,B_0}_0) d\PP(Y^{(\K)})=-
\sum_{i\in I}\mu(i) P(i,\E) H(D^i);  
\end{eqnarray*}
and
\begin{eqnarray}
\nonumber 
 h(\aY^{(\A)}) &=&-(1-\gamma)\sum_{j\in I}
\mu(j)\log \mu(j)-\sum_{i,j\in I} \mu(i) P(i,j)\log P(i,j)\\
\label{equ20}
&{}&\; -\sum_{i\in I, \delta\in \E}\mu(i) P(i,\delta) \log
P(i,\delta).
\end{eqnarray}
\end{proposition}

\begin{proof}
From (\ref{equ18}) and (\ref{equ19}) and by using the Markov property 
one gets 
\begin{eqnarray*}
\Delta(D)&=&\int H_{\aP_{Y^{(\A)}}}(D^{Y_0,B_0}_0) d\PP(Y^{(\K)})\\
&=&-\sum_{i\in I}\mu(i)P(i,\E) 
\sum_{\delta\in \E} P(i,\delta)/P(i,\E) \log (P(i,\delta)/P(i,\E))\\
&=& -\sum_{i\in I,\delta\in \E} \mu(i) P(i,\delta) \log
P(i,\delta)+\sum_{i\in I} \mu(i) P(i,\E) \log P(i,\E).
\end{eqnarray*}
Now, one uses (\ref{equ16}) to get expression (\ref{equ20}).
\end{proof}

\begin{remark}
\label{rmk7}
Let $\yX^{(a)}$ be a trajectory of an absorbed
chain, starting from distribution $\mu$ in $I$ and finishing after
it hits $\E$. It has length $\tau_\E$ and it corresponds to an
absorbed trajectory of length $\tau_\E-1$ in the chain 
$\aY^{\A}$ with alphabet $I^2\times \E^*$. In fact if 
$\yX^{(a)}=(X_1,..,X_l,\epsilon)$ with $X_1,..,X_l\in I$, 
$\epsilon\in \E$, is an absorbed trajectory of length $l+1$, 
then the associated trajectory in $\aY^{\A}$ is given
by the trajectories 
$Y^{(a)}=((X_r,X_{r+1},o), r=1,..,l-1; (X_{l},j^*,\epsilon))$
of length $l$.
Here $j^*\in I$ is an element chosen with distribution $\mu$ and
it is the starting state of the next absorbed trajectory. $\Box$
\end{remark}

\subsection{Entropy balance}

The associated stationary chain $\axX$ with transition kernel 
$\xP=P^\pi$ is retrieved from 
the stationary chain $\aY$ with transition kernel $Q$,
a collection of Bernoulli variables $\aB^{I,I}$ 
that assign $0$ or $1$ between the
connections of $\aY$, a set of Bernoulli variables 
$\aD^I$ giving the transition 
from $I$ to $\E$, and a family of walks $\aV$ whose components
are Bernoulli variables $(G_n)$ 
distributed as $\pi(\bullet \, | \, \E)$. The length of these 
walks is Geometric$(\gamma)-1$ distributed, and so they could be empty. 

\medskip

It is straightforward to check the following equality, relating
$h(\axX)$ given by (\ref{equ21}) to the entropies of the
elements forming the chain $\axX$. 

\begin{proposition}
\label{prop4}
We have
$$
h(\axX)=\pi(I) h(\aY^{(\A)})+\pi(\E)^2 h(\aG)
+\pi(I)\pi(\E) \log \pi(I) +\pi(\E)^2 \log \pi(\E). 
$$
$\Box$
\end{proposition}

Below we discuss the way this equality appears.
We have reduced the elements forming the chain $\axX$ to only two, 
the absorbed chains $\aY^{(\A)}$ and the walks $\aV$ with 
Bernoulli variables $G_n$. From (\ref{equ20}) we have
\begin{eqnarray*}
h(\aY^{(\A)})&=&-\pi(\E)\sum_{j\in I}
\mu(j)\log \mu(j)- \sum_{i,j\in I} \mu(i) P(i,j)\log P(i,j)\\
&{}& \, -\sum_{i\in I, \delta\in \E}\!\!\!\! \mu(i) P(i,\delta) \log
P(i,\delta),
\end{eqnarray*}
and the Bernoulli sequence $\aG=(G_n)$  has entropy
$$
h(\aG)=-\sum_{\delta\in \E} \pi(\delta \, | \, \E) \log
\pi(\delta \, | \, \E)=-\pi(\E)^{-1}
\sum_{\delta\in \E} \pi(\delta) \log \pi(\delta)+\log \pi(\E).
$$
Take big $N$, we divide the
sequence $(X_1,..,X_N)$ into the set of absorbed chains
$\xX^{(a)}$ and the set of nonempty walks $V$ in $\E$.
The proportion of elements in $I$ approaches $\pi(I)$ 
as $N\to \infty$. On the other hand, from (\ref{equ6})  
one obtains that for every time $t\in \tT$ there are in mean
$\sum_{i\in I} \mu(i) P(i,\E) (\pi(I)^{-1}-1)$ points belonging
to a walk in $\E$. Since,
$$
\sum_{i\in I} \mu(i) P(i,\E) (\pi(I)^{-1}-1)=\pi(\E)^2,
$$
one finds that the proportion of sites in $(X_1,..,X_N)$ with symbols 
in $\aG$ arising from a walk $V$ in $\E$ approaches to 
$\pi(\E)^2$ as $N\to \infty$. We have
\begin{equation}
\label{equ22}
\pi(\E)^2 h(\aG)=-\pi(\E) \sum_{\delta\in \E} \pi(\delta) \log \pi(\delta)
+\pi(\E)^2\log \pi(\E).
\end{equation}   

Let us compute $\pi(I) h(\aY^{(\A)})$. 
Since $\mu(i)=\pi(i)/\pi(I)$ for $i\in I$, 
one has
$$
-\pi(I)\sum_{j\in I} \mu(j)\log \mu(j)=
-\sum_{j\in I} \pi(j)\log \pi(j)+\pi(I)\log \pi(I),
$$
and so by using (\ref{equ20}) one obtains,
\begin{eqnarray*}
\pi(I) h(\aY^{(\A)})&=&
-\pi(\E)\sum_{j\in I}\! \pi(j)\log \pi(j)+\pi(\E) \pi(I)\log \pi(I)\\
&{}&\;\; -\sum_{i,j\in I}\! \pi(i) P(i,j)\log P(i,j)-\sum_{i\in I, \delta\in \E} 
\!\! \pi(i) P(i,\delta) \log P(i,\delta).
\end{eqnarray*}
Then, one has checked the equality given in Proposition \ref{prop4},
$$
h(\axX)-(\pi(\E)^2 h(\aG)+\pi(\E)^2 \log \pi(\E))=\pi(I)h(\aY^{(\A)})+
\pi(\E) \pi(I)\log \pi(I).
$$

The term $\pi(I)\pi(\E) \log \pi(I)$ has an origin similar 
to the one of last term in (\ref{equ22}). In fact, from
Remark \ref{rmk4}, the weights $\mu(j), j\in I$, appear with frequency 
$\pi(\E)$ in the sequence $\aY$ because this occurs 
at the sites where is a jump to $\E$. Since 
the sequence $\aY$ appears
with frequency $\pi(I)$, then with frequency $\pi(I)\pi(\E)$
it will appear the term $-\sum_{i\in I}\mu(i)\log \mu(i)$.
Hence, as in (\ref{equ22}) one has 
$$
-\pi(I)\pi(\E)\sum_{i\in I}\mu(i)\log \mu(i)=
-\pi(\E)\sum_{i\in I}\pi(i)\log \pi(i)
+\pi(I)\pi(\E) \log \pi(I),
$$
and since $-\pi(\E)\sum_{i\in I}\pi(i)\log \pi(i)$ 
is the term present in (\ref{equ21}), it appears 
the extra term $\pi(I)\pi(\E) \log \pi(I)$.

\begin{remark}
\label{rmk8}
From Remark \ref{rmk7} the length of an absorbed 
trajectory in $\aY^{\A}$ with alphabet $I^2\times \E^*$,  
has the same length as the number of elements in $I$ 
of an absorbed trajectory $\xX^{(a)}$   
starting from $\mu$ and absorbed when hitting $\E$ 
(this is of length $|\xX^{(a)}|-1$ which counts the visited 
sites in $I$, but not the one containing the absorbing state). 
Since the entropy of
a system is the gain of entropy per unit of time,
the proportion of symbols given the entropy $h(\aY^{(\A)})$ is 
$\pi(I)$.
This explains why it appears the term 
$\pi(I) h(\aY^{(\A)})$. $\Box$
\end{remark}

\section{Appendix: Proof of Proposition \ref{prop2}}

Since only a shift is needed to construct the stationary process 
$\hS^s$ from $\hS$, we get for any $ a,b\in I\cup \E$,
$$
\PP(\hS^s_{t+1}=b \, | \, \hS^s_{t}=a)=\PP(\hS_{t+1}=b \, | \,
\hS_{t}=a).
$$
Or, more generally, for any $b\in I\cup \E$ and any sequence of symbols
$a(u\le t)=(a(u)\in I\cup \E: u\le t)$, one has
$$
\PP(\hS^s_{t+1}=b \, | \, \hS_{u}=a(u), u\le t)=\PP(\hS_{t+1}=b
\, | \, \hS_{u}=a(u), u\le t).
$$
Since there is regeneration at the times in 
$\tT=\{t\in \ZZ: \hS_t\in I\}$, one gets
\begin{equation}
\label{equ23}
\PP(\hS^s_{t+1}=b \, | \, \hS^s_{u}=a(u), u\le t)=
\PP(\hS^s_{t+1}=b \, | \, \hS^s_{u}=a(u) , u=t,..,t-r)
\end{equation}
being $r\ge 0$ the first nonnegative element such that $a(t-r)\in I$.

\medskip

Let us compute $\PP(t\in \tT)$. It suffices to calculate $\PP(t\not\in
T)/\PP(t\in \tT)$. From (\ref{equ6}) one gets that for every
time $t\in \tT$ there are in mean
$$
\sum_{i\in I}\mu(i)
\sum_{j\in I}Q(i,j) \theta_{i,j} \pi(I)^{-1}
=\sum_{i\in I} \mu(i) P(i,\E) \pi(I)^{-1}
$$
points in $\ZZ\setminus \tT$. Then, from (\ref{equ4}) one finds
$$
\PP(t\not\in \tT)/\PP(t\in \tT)=\sum_{i\in I} \mu(i) P(i,\E) \pi(I)^{-1}
= \pi(\E)/\pi(I).
$$
We conclude
\begin{equation}
\label{equ24}
\PP(t\in \tT)=\pi(I) \hbox{ and } \PP(t\not\in \tT)=\pi(\E).
\end{equation}
Since $(\hS_{T_l}: l\in \ZZ)$ is equally distributed as
$(Y_n: n\in \ZZ)$ one gets for $i\in I$,
$$
\PP(\hS^s_t=i \, | \, \hS^s_t\in I)=
\PP(\hS_t=i \, | \, \hS_t\in I)=\mu(i),
$$
and so, by using (\ref{equ24}) one gets  
\begin{equation}
\label{equ25}
\PP(\hS^s_t=i)=\mu(i) \PP(\hS_t\in I)=\mu(i)\pi(I)=\pi(i).
\end{equation}  
Let $i,j\in I$, from definition of $\theta_{i,j}$ in (\ref{equ7})
and since $\oth_{i,j}=1-\theta_{i,j}$ one obtains
\begin{eqnarray}
\label{equ26}
\PP(\hS^s_{t+1}=j \, | \, \hS^s_t=i)&=& \oth_{i,j}Q(i,j)\\
\nonumber
&=&
\oth_{i,j}(P(i,j)+P(i,\E)\mu(j))=P(i,j).
\end{eqnarray}  
We have $\sum_{j\in I}\PP(\hS^s_{t+1}=j \, | \, \hS^s_t=i)  
=1-P(i,\E)$, so
$$
\PP(\hS^s_{t+1}\in \E \, | \, \hS^s_t=i)=P(i,\E).
$$
Then, when $\hS^s_t=i$ jumps to $\E$ it does it with
probability $P(i,\E)$, and the jump to some particular state
$\delta\in \E$ is done with probability
\begin{equation}
\label{equ27}
\PP(\hS^s_{t+1}=\delta \, | \, \hS^s_t=i)=
P(i,\E)P(i,\delta)/P(i,\E)=P(i,\delta).
\end{equation}
For $\delta\in \E$, $\epsilon\in \E$ one has
$\xP(\epsilon,\delta)=\pi(\delta)$, then
$$
\PP(\hS^s_{t+1}=\epsilon \, | \, \hS^s_t=\delta)=
\pi(\epsilon \, | \, \E)\PP_\delta(\tau_I>1)=
\pi(\epsilon \, | \, \E)\pi(\E)=\pi(\epsilon).
$$
Then, for $\delta\in \E$ one finds
\begin{eqnarray*}
\PP(\hS^s_t=\delta)&=&\sum_{i\in I}\PP(\hS^s_{t-1}=i, \hS^s_t=\delta)
+\sum_{\epsilon\in \E}\PP(\hS^s_{t-1}=\epsilon, \hS^s_t=\delta)\\
&=&\sum_{i\in I}\PP(\hS^s_t=\delta \, | \, \hS^s_{t-1}=i)\pi(i)+
\sum_{\epsilon\in \E}\PP(\hS^s_t=\delta \, | \, \hS^s_{t-1}=\epsilon)
\PP(\hS^s_{t-1}=\epsilon)\\
&=&\sum_{i\in I}P(i,\delta)\pi(i)+\pi(\delta)\sum_{\epsilon\in \E}
\PP(\hS^s_{t-1}=\epsilon)\\
&=&\sum_{i\in I}P(i,\delta)\pi(i)+
\pi(\delta)\pi(\E)=\pi(\delta)(\pi(I)+\pi(\E))=\pi(\delta)\,.
\end{eqnarray*}
By definition of the process $\hS$, for $\epsilon\in \E$ one gets
\begin{equation}
\label{equ28}
\PP(\hS^s_{t+1}=\epsilon \, | \, \hS^s_t=\delta, \hS^s_{u}=a(u), u<t)
=\PP(\hS^s_{t+1}=\epsilon \, | \, \hS^s_t=\delta)
=\pi(\epsilon).
\end{equation}
Now, let us compute $\PP(\hS^s_t=\epsilon, \hS^s_{t+1}=j)$
for $\epsilon\in \E$, $j\in I$. We necessarily have that
this sequence has its origin in some $(Y_s=i, Y_{s+1}=j)$
and $B^{i,j}_{s+1}=0$, for some $i\in I$. Then, by summing over  
all the states $i\in I$ and all pieces of trajectories in $\E$
that are built between $i$ and $j$, and by using (\ref{equ27})   
and (\ref{equ5}) one gets
\begin{eqnarray*}
&{}&\PP(\hS^s_t=\epsilon, \hS^s_{t+1}=j)\\
&{}& = \sum_{i\in I} \sum_{l\ge 1} \PP(\hS_{t-l}=i;
\hS_{t-u}\in \E, 1\le u< l; \hS^s_t=\epsilon, \hS^s_{t+1}=j)\\
&=&
\sum_{i\in I} \mu(i) \mu(j)  P(i,\epsilon) \pi(I)+
\! \sum_{i\in I}\mu(i) \mu(j) 
\!\!\!\! \sum_{l\ge 1; \delta_1,..,\delta_l\in \E} 
\!\!
\left(\! P(i,\delta_1)\prod_{k=2}^{l-1}
\pi(\delta_k)\! \right)\! \pi(\epsilon)\pi(I)\\
&=&\pi(\epsilon)\pi(I)\mu(j)=\pi(\epsilon)\pi(j).
\end{eqnarray*}
From (\ref{equ25}) and (\ref{equ28}) the bivariate marginals of
the stationary chains $\axX=(\xX_n)$ and ${\bf{\hS}^s}=(\hS^s_n)$
are the same. Now we turn to prove that ${\bf{\hS}^s}$ satisfies
the Markov property. In view of the regeneration
property (\ref{equ23}), this will be proven once we show
\begin{equation}
\label{equ29}
\PP(\hS^s_{t+1}=b \, | \, \hS^s_{u}=a(u) , u=t,..,t-r)=
\PP(\xX_{t+1}=b \, | \, \xX_{u}=a(t))
\end{equation}
where $r\ge 0$ and satisfies $a(t-r)\in I$, $a(t-u)\in \E$
for $u=1,..,r-1$. This was shown for the case $r=0$ in
(\ref{equ26}) and (\ref{equ27}). On the other hand (\ref{equ28})
proves (\ref{equ29}) in the case $b\in \E$ and $r>0$. So, the
unique case left to show is $b\in I$ and $r>0$.

\medskip

Let $i,j\in I$, $r>0$, $\delta_u\in \E$ for $u=0,..,r-1$. Since
$$
\PP(\xX_{t+1}=j \, | \, \xX_{t-u}=\delta_0)=\pi(j),
$$
to achieve the proof of (\ref{equ29}), the unique relation
that we are left to show is
$$
\PP(\hS^s_{t+1}=j \, | \, \hS^s_{t-u}=\delta_u, u=0,..,r-1,
\hS^s_{t-r}=i)=\pi(j).
$$
Let us prove it. One has
\begin{eqnarray*}
&{}&
\PP(\hS^s_{t+1}=j, \hS^s_{t-u}=\delta_u, u=0,..,r-1, \hS^s_{t-r}=i)\\
&=&
\mu(i)(P(i,j)+P(i,\E)\mu(j))\theta_{i,j}\frac{P(i,\delta_0}{P(i,\E)}
\left(\prod_{u=1}^{r-1}\pi(\delta_l)\right)\pi(I)\\
&=&
\mu(i)P(i,\E)\mu(j)\frac{P(i,\delta_0}{P(i,\E)}
\left(\prod_{u=1}^{r-1}\pi(\delta_l)\right)\pi(I)\\
&=&\mu(i)P(i,\delta_0)\mu(j)
\left(\prod_{u=1}^{r-1}\pi(\delta_l)\right)\pi(I), 
\end{eqnarray*}
and
\begin{eqnarray*}
\PP(\hS^s_t=i, \hS^s_{t-u}=\delta_u: u=0,..,r-1)
&=& \left(\sum_{j\in I}\mu(i)P(i,\delta_0)\mu(j)\right)
\left(\prod_{u=1}^{r-1}\pi(\delta_l)\right)\\
&=&\mu(i)P(i,\delta_0)
\left(\prod_{u=1}^{r-1}\pi(\delta_l)\right).
\end{eqnarray*}
Therefore,
$$  
\PP(\hS^s_{t+1}=j\, | \, \hS^s_{t-u}=\delta_u, u=0,..,r-1, \hS^s_{t-r}=i)
=\mu(j)\pi(I)=\pi(j).
$$
Then (\ref{equ29}) follows. We have proven that the laws of the
stationary chains $\axX=(\xX_n)$ and ${\bf{\hS}^s}=(\hS^s_n)$ are the same.

\bigskip

\noindent \textbf{Acknowledgments:} This work was supported by the Basal
Conicyt project AFB170001. The author thanks Dr. Michael Schraudner from
CMM, University of Chile,  
for calling my attention to reference \cite{dos}.

\noindent SERVET MART\'INEZ

\noindent {\it Departamento Ingenier{\'\i}a Matem\'atica and Centro
Modelamiento Matem\'atico, Universidad de Chile,
UMI 2807 CNRS, Casilla 170-3, Correo 3, Santiago, Chile.}
e-mail: smartine@dim.uchile.cl


\medskip

\begin{thebibliography}{99}

\bibitem{ar} L. Abramov, V. Rokhlin. 
The entropy of a skew product of measure-
preserving transformations. 
Vestnik Leningrad. Univ. 17 (1962), 5–13 (in Russian).

\bibitem{as} S. Asmussen.
Applied probability and queues. Second edition. Applications of 
Mathematics (New York), 51. Stochastic Modelling and Applied 
Probability. Springer-Verlag, New York (2003).

\bibitem{cmsm}
P. Collet, S. Mart\'inez, J. San Mart\'in. Quasi-stationary
distributions. Markov chains, diffusions and dynamical systems.   
Probability and its Applications (New York). Springer, Heidelberg
(2013).

\bibitem{das} J. Darroch, E. Seneta. 
On quasi-stationary distributions in absorbing 
discrete-time finite Markov chains. 
J. Appl. Probab. 2 (1965), 88–100.

\bibitem{dgs} M. Denker, C. Grillenberger, K Sigmund. Ergodic 
theory on compact spaces. Lecture Notes in Mathematics, Vol. 527. 
Springer-Verlag, Berlin-New York
(1976). 

\bibitem{dos} T. Downarowicz, J. Serafin. Fiber entropy and 
conditional variational principles in compact non-metrizable spaces. 
Fund. Math. 172 (2002), No. 3, 217–247. 

\bibitem{fmp}
P.A. Ferrari, S. Mart\'inez, P. Picco. 
Existence of Non-Trivial Quasi-Stationary 
Distributions in the Birth-Death Chain. 
Adv. Appl. Probab. 24 (1992), No. 4, pp. 795-813

\bibitem{fkmp}
P.A. Ferrari, H. Kesten, S. Mart\'inez, P. Picco. Existence of 
quasi-stationary distributions. A renewal dynamical approach. Ann. 
Probab. 23 (1995), No. 2, 501-521.

\bibitem{ks}
M. Keane, M. Smorodinsky. Finitary isomorphisms of irreducible
Markov shifts. Israel J. of Math. 34 (1979), 
no. 4, 281-286.

\end{thebibliography}
\end{document}